\documentclass[12pt]{amsart}
\usepackage{amssymb,amscd}
\usepackage{stmaryrd}
\usepackage{graphicx}
\usepackage[all]{xy}
\begin{document}
\newtheorem{thm}{Theorem}[section]
\newtheorem{lem}[thm]{Lemma}
\newtheorem{prop}[thm]{Proposition}
\newtheorem{cor}[thm]{Corollary}
\theoremstyle{definition}
\newtheorem{ex}[thm]{Example}
\newtheorem{rem}[thm]{Remark}
\newtheorem{prob}[thm]{Problem}
\newtheorem{thmA}{Theorem}
\renewcommand{\thethmA}{}
\newtheorem{defi}[thm]{Definition}
\renewcommand{\thedefi}{}
\input amssym.def
\long\def\alert#1{\smallskip{\hskip\parindent\vrule%
\vbox{\advance\hsize-2\parindent\hrule\smallskip\parindent.4\parindent%
\narrower\noindent#1\smallskip\hrule}\vrule\hfill}\smallskip}
\def\ff{\frak}
\def\Spec{\mbox{\rm Spec}}
\def\type{\mbox{ type}}
\def\Hom{\mbox{ Hom}}
\def\rank{\mbox{ rank}}
\def\Ext{\mbox{ Ext}}
\def\Ker{\mbox{ Ker}}
\def\Max{\mbox{\rm Max}}
\def\End{\mbox{\rm End}}
\def\l{\langle\:}
\def\r{\:\rangle}
\def\Rad{\mbox{\rm Rad}}
\def\Zar{\mbox{\rm Zar}}
\def\Supp{\mbox{\rm Supp}}
\def\Rep{\mbox{\rm Rep}}
\def\cal{\mathcal}
\title[CHMV-algebras]{Compact Hausdorff MV-algebras: Structure, Duality and Projectivity}
\thanks{\today}
\author{Jean B Nganou}
\address{Department of Mathematics and Statistics, University of Houston-Downtown,
Houston, TX 77001} \email{nganouj@uhd.edu}
\begin{abstract} It is proved that the category $\mathbb{EM}$ of extended multisets is dually equivalent to the category $\mathbb{CHMV}$ of compact Hausdorff MV-algebras with continuous homomorphisms, which is in turn equivalent to the category of complete and completely distributive MV-algebras with homomorphisms that reflect principal maximal ideals. Urysohn-Strauss's Lemma, Gleason's Theorem, and projective objects are also investigated for topological MV-algebras.
 \vspace{0.20in}\\
{\noindent} Key words: MV-algebra, compact MV-algebras, multiset, dually equivalent, maximal ideal, completely distributive, extremally disconnected, projective MV-algebra.
\end{abstract}
\maketitle
\section{Introduction}
    One notable duality from the theory of Boolean algebras is the one between sets with functions and atomic complete Boolean algebras with complete homomorphisms. Given that MV-algebras which constitute the algebraic counterpart of \L ukasiewicz infinite-valued logic \cite{CC} are extensions of Boolean algebras, there have been several extension of the above dually to notable subvarieties of MV-algebras. For instance, it has been established that the category of locally finite MV-algebras is dually equivalent to the category of generalized multisets \cite{CDM}, a duality that was very recently extended further to weakly locally finite MV-algebras \cite{CM}, and that the category of profinite MV-algebras and homomorphisms that reflect principal maximal ideals is dually equivalent to the category of multisets \cite{jbn}. Multisets are defined in combinatorics as pairs $\langle X, \sigma\rangle$, where $X$ is a set and $\sigma:X\to \mathbb{N}$ assigning to each $x$ its multiplicity $\sigma(x)$. There have been several variants and generalizations of the concept of multisets (see for e.g., \cite{Bli, CDM, H, Mo}).\par
A topological MV-algebra is an MV-algebra $(A,\oplus, \neg, 0)$ together with a topology $\tau$ with respect to which the operations $\oplus$ and $\neg$ are continuous. In addition if the topology $\tau$ has the property $\mathcal{P}$ (such as compactness, Hausdorff, connectedness, etc...), one refers to $A$ as a $\mathcal{P}$ MV-algebras. For basic properties of topological MV-algebras, we refer the reader to \cite{Hoo, web}.\par
The main results of the paper consist on one hand to establish a duality between the category of compact Hausdorff MV-algebras and continuous homomorphisms and the category of extended multisets and their morphisms, and describe projective objects of the first category on the other. \par
One of the main ingredients used in obtaining the duality is the characterization of principal maximal ideals of compact Hausdorff MV-algebras as being exactly the compact maximal ideals. The category of compact Hausdorff MV-algebras overlaps significantly with that of locally finite MV-algebras, and also that of weakly locally finite MV-algebras, and contains the category of Stone MV-algebras as a full subcategory. The duality obtained here extends (at least partially) those established in \cite{CDM, CM, jbn}.\par
Some of the most prominent characteristics of normal spaces are about the existence of continuous maps subject to some type of separation, the best known being the Urysohn's Lemma. Since compact Hausdorff are normal, many versions of this lemma have been considered in the context of topological algebras, where continuous functions would also be required to preserve the algebras structures. For instance, the Urysohn-Strauss Lemma for compact Hausdorff distributive lattices (see for e.g., \cite[Lem. VII.1.14, Thm. VII.1.14]{Jo}. Since the main objects of study in this article are compact Hausdorff MV-algebras, one would like to know if the of the continuous lattice homomorphisms that exist, there is some that preserve the MV-algebra structures. We obtain that the answer is positive if and only if the compact Hausdorff MV-algebra is a compact Hausdorff Boolean algebra, i.e., a powerset algebra\cite{GH}.\par
Extremally disconnected topological spaces are spaces in which the closures of open subsets are open\cite{GL}. They are also known as Stonean spaces since they are exactly the Stone spaces for which the Boolean algebra of clopen subsets is complete\cite[Prop. III. 3.4]{Jo}. Above all, these spaces are best known for being the projective objects in the category of compact Hausdorff spaces with continuous maps as proved by Gleason \cite[Thm. 2.5]{GL}. In the present context, we show that the extremally disconnected topological MV-algebras are finite MV-algebras, and do not coincide with the projective objects in the category of compact Hausdorff MV-algebras with continuous homomorphisms. Indeed, we describe all projective compact Hausdorff MV-algebras and obtain that they are exactly the CHMV-algebras having the 2-element Boolean algebra as a continuous homomorphic image. Projective objects have been investigated in many varieties of distributive lattices\cite{BA, BH}, including some varieties of MV-algebras \cite{DNG, DNGL}.\\
We set up the notations and terminologies used in the paper.\\
The prototype of MV-algebra is the unit real interval $[0,1]$ equipped with the operation of truncated addition $x\oplus y=\text{min}\{x+y,1\}$, negation $\neg x=1-x$, and the identity element $0$. For each integer $n\geq 2$, $\L_n=[0,1]\cap \mathbb{Z}\dfrac{1}{n-1}$ is a sub-MV-algebra of $[0,1]$ (the \L ukasiewicz's chain with $n$ elements), and up to isomorphism every finite MV-chain is of this form. For convenience and uniformity we will also denote $[0,1]$ by $\L_{\infty}$.

We assume familiarity with MV-algebras, in particular their definition, homomorphisms, prime ideals and maximal ideals\cite{C2}.\\
The set of natural numbers extended by $\infty$ will be denoted by $\overline{\mathbb{N}}$, that is $\overline{\mathbb{N}}=\mathbb{N}\cup \{\infty\}$.\par
The main object of study in this work are compact Hausdorff MV-algebras (CHMV-algebras, for short), which are MV-algebras equipped with a compact and Hausdorff topology with respect to which all MV-operations are continuous. These are known up to algebraic and topological isomorphism (see for e.g., \cite[Theorem 2.2]{jbn2} or \cite[Theorem 2.5]{web}) to be of the form $A:=\displaystyle\prod_{x\in X}\L_{n_x}$ with $n_x\in \overline{\mathbb{N}}$ for all $x\in X$. We set $X^A_{\text{fin}}=\{x\in X:n_x\in\mathbb{N}\}$ and $X^A_{\infty}:=\{x\in X:n_x=\infty\}$. We will simply write $X_{\text{fin}}$ and $X_{\infty}$ when there is no risk of confusion.  It follows that $A\cong A_{\text{fin}} \times A_{\infty}$, where $A_{\text{fin}} = \prod_{x\in X^A_{\text{fin}}}\L_{n_x}$ and $A_{\infty}= [0,1]^{X^A_{\infty}}$.

For each $x\in X$, $p_x:A\to L_{n_x}$ denotes the natural projection. In addition $ker p_x$ will be denoted by $M_x$. In particular, it follows that each $M_x$ is a maximal ideal of $A$. It is easy to see that $$\displaystyle\oplus_{x\in X}\L_{n_x}:=\left\{f\in A:f(x)=0\; \text{for all, but finitely many}\;  x\in X \right \}$$ is an ideal of $A$. 
\section{Maximal ideals of CHMV-algebras}
For every MV-algebra $A$, let $\mathcal{H}(A)$ denotes the set of MV-algebra homomorphisms from $A$ into $[0,1]$ and $Max(A)$ denotes the set of maximal ideals of $A$. It is well known \cite{CDM} that $\chi\mapsto ker\chi$ defines a one-to-one correspondence between $\mathcal{H}(A)$ and $Max(A)$, where $ker \chi=\{a\in A: \chi(a)=0\}$. \\

Recall that a principal ideal of an MV-algebra $A$ is any ideal $I$ that is generated by a single element, that is there exists $a\in A$, such that $I=\langle a\rangle$. It is well known that $x\in \langle a\rangle$ if and only if $x\leq na$ for some integer $n\geq 1$. 

While our main class of MV-algebras of interest is the class of class of compact Hausdorff MV-algebras, we consider a slightly larger class since some of our results in this section hold in the larger class. The class in question is that of MV-algebras that isomorphic to direct products of simple MV-chains or sub-MV-algebras of $[0,1]$. We shall called such MV-algebras \textit{strictly semisimple}. By the definition, each strictly semisimple MV-algebra has the form $A:=\prod_{x\in X}A_x$ with $A_x$ sub-MV-algebra of $[0,1]$ for each $x\in X$. Then $A$ has a natural Hausdorff topology, namely the product topology $\mathfrak{p}$ of the natural topologies on the $A_x$'s and $(A, \mathfrak{p})$ is a topological MV-algebra. Moreover, this topology is compact if and only if each $A_x$ is either a finite MV-chain or $[0,1]$. In this case $\mathfrak{p}$ is the only compact Hausdorff topology making $A$ a topological MV-algebra \cite[Lemma 2.1]{jbn2}, and for this reason, unless otherwise specified, whenever a topology is used on a CHMV-algebra, it would be referring to this topology.


The following result generalizes the corresponding one obtained in \cite[Lemma 3.1]{jbn} for Stone MV-algebras i.e., topological MV-algebras whose topology is Stone (compact Hausdorff and zero-dimensional).
\begin{prop}\label{principals}
Let $A:=\displaystyle\prod_{x\in X}A_x$ be a strictly semisimple MV-algebra. For every maximal ideal $M$ of $A$, the following conditions are equivalent.
\begin{enumerate}
\item $M$ is principal;
\item There exists a unique $x_0\in X$, such that $M=ker p_{x_0}$;
\item $M$ does not contain $\oplus_{x\in X}A_x$;
\item $\bigvee M\in M\cap B(A)$;
\end{enumerate}
\end{prop}
\begin{proof}
$(1)\Rightarrow (2)$: Suppose that $M$ is principal, then $M=\langle a\rangle $ for some $a\in A$. We claim that there exists $x_0\in X$ with $a(x_0)=0$. By contradiction suppose that $a(x)\ne 0$ for all $x\in X$, and let $\alpha=\text{Inf}\{a(x):x\in X\}$.\\
 If $\alpha>0$, then there exits an integer $m\geq 1$ such that $1\leq m\alpha$, which implies $1\leq ma(x)$ for all $x\in X$. It follows that $ma=1$, and so $M=A$, which would be a contradiction. Therefore $\alpha=0\notin \{a(x):x\in X\}$. One can write $X$ as the disjoint union of of two sets $X'$ and $X''$ such that $\text{Inf}\{a(x):x\in X'\}=\text{Inf}\{a(x):x\in X''\}=0$. Consider $f,g\in A$ defined by:
$$f(x)=
\left\{\begin{array}{ll}
  1 & ,\ \ \mbox{if} \ \ x\in X'\\
 a(x)&  ,\ \ \mbox{if} \ \ x\in X''
\end{array}\right.\; \; \; 
\text{and}\; \; \; 
g(x)=
\left\{\begin{array}{ll}
  a(x) & ,\ \ \mbox{if} \ \ x\in X'\\
 1&  ,\ \ \mbox{if} \ \ x\in X''
\end{array}\right.$$
Then $f\wedge g=a$, in particular $f\wedge g\in M$. Since $M$ is prime, as every maximal ideal is, then $f\in M$ or $g\in M$. Assume $f\in M=\langle a\rangle$, then there exists an integer $r\geq 1$ such that $f\leq ra$. Therefore, $1\leq ra(x)$ for all $x\in X'$, and so $1/r\leq a(x)$ for all $x\in X'$. This contradicts the fact that $\text{Inf}\{a(x):x\in X'\}=0$. In a similar argument, $g\in M$ would contradict the fact that $\text{Inf}\{a(x):x\in X''\}=0$.\\
Thus $a(x_0)=0$ for some $x_0\in X$. For every $f\in M=\langle a\rangle$, there exists $k\geq 1$ such that $f\leq ka$, and it follows that $f(x_0)=0$ for all $f\in M$. Hence, $M\subseteq M_{x_0}$. Since $M$ and $M_{x_0}$ are maximal, then $M=M_{x_0}=\ker p_{x_0}$. The uniqueness is clear.\\
$(2)\Rightarrow (1)$: This is clear as each $M_{x_0}$ is principal and generated by $f$, where $$f(x)=
\left\{\begin{array}{ll}
  0 & ,\ \ \mbox{if} \ \ x=x_0\\
 1&  ,\ \ \mbox{if} \ \ x\ne x_0
\end{array}\right.$$

$(2) \Rightarrow (3)$: Suppose that there exists a unique $x_0\in X$ such that $M=\ker p_{x_0}$. Consider $f\in A$ defined by $$f(x)=
\left\{\begin{array}{ll}
  1 & ,\ \ \mbox{if} \ \ x=x_0\\
 0&  ,\ \ \mbox{if} \ \ x\ne x_0
\end{array}\right. $$
Then $f\in \displaystyle\oplus_{x\in X}\L_{n_x}$ and $f\notin M$.\\
$(3)\Rightarrow (2)$: Suppose that for all $x\in X$, $M\ne M_x$. For each $x\in X$, let $b_x\in A$ be defined by $$b_x(t)=
\left\{\begin{array}{ll}
  0 & ,\ \ \mbox{if} \ \ t=x\\
 1&  ,\ \ \mbox{if} \ \ t\ne x
\end{array}\right.$$
Then for every $x\in X$, since $M_x=\langle b_x\rangle$, then $b_x\notin M$ and since $M$ is maximal, by \cite[Proposition 1.2.2]{C2} there exists an integer $k_x\geq 1$ such that $\neg k_xb_x=\neg b_x\in M$. Now, let $a\in \oplus_{x\in X}\L_{n_x}$, then there exists $k\geq 1$ and $x_1, x_2, \ldots, x_k\in X$ such that $a\leq \neg b_{x_1}\oplus \neg b_{x_2}\oplus \cdots \oplus \neg b_{x_k}$. Hence, as $M$ is a lower set, $a\in M$ and $\oplus_{x\in X}\L_{n_x}\subseteq M$ as needed.\\
$(2)\Rightarrow (4)$: For each $x\in X$, $\bigvee M_x=b_x$ where $b_x(x)=0$ and $b_x(t)=1$ if $t\ne x$. So $\bigvee M_x\in M_x\cap B(A)$. \\
$(4)\Rightarrow (3)$: By contradiction suppose that $\bigvee M\in M\cap B(A)$ and $\oplus_{x\in X}\L_{n_x}\subseteq M$. Then $1=\bigvee \oplus_{x\in X}\L_{n_x}\leq \bigvee M$. So $1=\bigvee M\in M$, which contradicts the fact that $M$ is a proper ideal of $A$.
\end{proof}
\begin{rem}
The equivalence $(1) \Leftrightarrow (2)$ of the preceding Proposition can also be derived from the fact that the maximal spectral space of $A$ is the Stone \v{C}ech compactification $\beta X$ of $X$ endowed with the discrete topology. Indeed $\mathfrak{m}\mapsto \mathfrak{u}_{\mathfrak{m}}:==\{S\subseteq X: f^{-1}([0,\epsilon))\subseteq S \; \text{for some}\; f\in \mathfrak{m} \; \text{and}\; \epsilon>0\}$ defines a homeomorphism from $\text{Max}(A)$ onto $\beta X$ and $\mathfrak{m}$ is principal if and only if $\mathfrak{u}_{\mathfrak{m}}$ is a principal.
\end{rem}
We deduce characterizations of non-principal maximal ideals.
\begin{cor}
Let $A:=\displaystyle\prod_{x\in X}A_x$ be a strictly semisimple MV-algebra and $M$ a maximal ideal of $A$. The following assertions are equivalent.
\begin{enumerate}
\item $M$ is non-principal;
\item $\oplus_{x\in X}A_x\subseteq M$
\item $M$ is dense in $A$, i.e., $\overline{M}=A$.
\end{enumerate}
\end{cor}
It also follows that the representation of a strictly semisimple MV-algebra is unique up to permutation of factors.
\begin{cor}\label{uniquerepr}
Let $A:=\prod_{x\in X}A_x$ and $B:=\prod_{y\in Y}B_y$ be two strictly semisimple MV-algebras. If $A$ is isomorphic to $B$, then there exists a bijection $\tau:X\to Y$ such that for all $x\in X$, $A_x=B_{\tau (x)}$.
\end{cor}
\begin{proof}
Let $A:=\displaystyle\prod_{x\in X}A_x$ and $B=\displaystyle\prod_{y\in Y}B_y$, and $\varphi:A\to B$ be an isomorphism. For each $x\in X, y\in Y$, denote by $p_x$ (resp. $q_y$) the natural projection $A\to A_x$ (resp. $B\to B_y$). Let $PMax(A)$ and 
$PMax(B)$ denote the set of principal maximal ideals of $A$ and $B$ respectively. Let $x\in X$, then $M_x\in PMax(A)$, so $\varphi(M_x)\in PMax(B)$, and by Proposition \ref{principals}, there exists a unique $y\in Y$ such that $\varphi(M_x)=ker q_y$.
 If one defines $\tau:X\to Y$ by $\tau (x)=y$, it is readily seen that $\tau$ is a bijection.  For each $x\in X$, consider the map $a\mapsto \varphi(a)/\varphi(M_x)$, which is clearly a surjective homomorphism from $A\to B/\varphi(M_x)$, whose kernel is 
 $M_x$. By the homomorphism theorem, we obtain that $A_x\cong A/M_x\cong B/ker q_y\cong B_y$. Thus, by Corrolary \cite[Corollary 7.2.6]{C2}, $A_x=B_{\tau(x)}$ for all $x\in X$ as required.
\end{proof}
It follows from the preceding facts that each strictly semisimple is completely determined by its principal simple quotients, i.e., its quotients by principal maximal ideals.
\begin{cor}
An MV-algebra is strictly semisimple if and only if it is isomorphic to its principal simple quotients.
\end{cor}
For CHMV-algebras, the characterizations of principal maximal ideals obtained in Proposition \ref{principals} can be expanded by a topological property.
\begin{prop}\label{principal}
A maximal ideal of a CHMV-algebra is principal if and only if it is compact.
\end{prop}
\begin{proof}
Let $A:=\displaystyle\prod_{x\in X}\L_{n_x}$ ($2\leq n_x\leq \infty$) be a CHMV-algebra. If $M$ is a principal maximal ideal of $A$, then by Proposition \ref{principals}, $M=\ker p_{x_0}=p_{x_0}^{-1}(\{0\})$, and $p_{x_0}$ is continuous and $\L_{n_{x_0}}$ is Hausdorff, then $M$ is closed in $A$. But since $A$ is compact, then $M$ is compact.\\
Conversely, suppose that $M$ is a compact maximal ideal of $A$. It is enough by Proposition \ref{principals} to prove that $\oplus_{x\in X}\L_{n_x}\nsubseteq M$. Assume by contradiction that $\oplus_{x\in X}\L_{n_x}\subseteq M$. Since $M$ is compact and $A$ is Hausdorff, then $M$ is closed and by \cite[Theorem VII.1.6]{Jo}, $M$ is closed under directed joins. Note that $\oplus_{x\in X}\L_{n_x}$ is a directed subset of $M$, so $1={\bf \setminus\kern-.95ex\nnearrow} (\oplus_{x\in X}\L_{n_x})\in M$. That is $1\in M$, which contradicts the fact that $M$ is a proper ideal.
\end{proof}
For CHMV-algebras, Corollary \ref{uniquerepr} can be strengthened as follows.
\begin{cor} \label{uniquerep}
The representation of a CHMV-algebras as product of complete MV-chains is unique up to a permutation of factors. In other words, if $$\displaystyle\prod_{x\in X}\L_{n_x}\cong\displaystyle\prod_{y\in Y}L_{m_y}$$ 
Then, there exists a bijection $\tau :X\to Y$ such that $m_{\tau(x)}=n_x$ for all $x\in X$.
\end{cor}
\begin{prop}
Let $A$, $B$ be two isomorphic CHMV-algebras. Then $A_{\text{fin}} \cong B_{\text{fin}}$ and $A_{\infty} \cong B_{\infty}$.  In particular, if $A$ is a Stone MV-algebra if and only if $B$ is a Stone MV-algebra.
\end{prop}
\begin{proof}
Suppose that $A:=\displaystyle\prod_{x\in X}\L_{n_x}$ and $B=\displaystyle\prod_{y\in Y}\L_{m_y}$, and $\varphi:A\to B$ an isomorphism. By Corollary \ref{uniquerep} there exists a bijection $\tau :X\to Y$ such that $m_{\tau(x)}=n_x$ for all $x\in X$. It follows that $\tau: X_{\text{fin}}\to Y_{\text{fin}}$ is a bijection and $m_{\tau(x)}=n_x$ for all $x\in X_{\text{X}}$. Therefore, $A_{\text{fin}} \cong B_{\text{fin}}$. It is also true that $A_{\infty} \cong B_{\infty}$ since $\tau:X_{\infty}\to Y_{\infty}$ is also a bijection.

If $A$ is a Stone MV-algebra, then by \cite[Theorem 2.3]{jbn2} $A:=\displaystyle\prod_{x\in X}\L_{n_x}$ with $2\leq n_x<\infty$ for all $x\in X$. Therefore, $X_{\infty}=\emptyset$ and it follows that $Y_{\infty}=\emptyset$. Hence, $B$ is a Stone MV-algebra.
\end{proof}
\begin{prop}\label{comp-h}
Let $A=\prod_{x\in X}\L_{n_x}$, $B=\prod_{y\in Y}\L_{m_y}$ be two CHMV-algebras and $\varphi: A\to B$ be a homomorphism. \\
Then the following conditions are equivalent:
\begin{itemize}
\item[(1)] For every $y\in Y$, there exists a unique $x\in X$ such that $q_y\circ \varphi=p_x$, where $p_x:A\to \L _{n_x}$, $q_y:B\to \L_{m_y}$ be the natural projections;
\item[(2)] $\varphi$ reflects principal maximal ideals (i.e., if $M$ is a principal maximal ideal of $B$, then $\varphi^{-1}(M)$ is a principal maximal ideal of $A$);
\item[(3)] $\varphi$ is complete (i.e., $\varphi$ preserves arbitrary suprema and infima);
\item[(4)] $\varphi$ is continuous.
\end{itemize}
\end{prop}
\begin{proof}
We prove that $(1)\Leftrightarrow (2)$. Suppose that (1) holds, and let $M$ be a principal maximal ideal of $B$. Then by Lemma \ref{principal}, there exists $y\in Y$ such that $M=ker q_y$. It follows from (1) that there exists a unique $x\in X$ such that $q_y\circ \varphi=p_x$. Hence, $ker p_x\subseteq \varphi^{-1}(M)$ and by maximality, we obtain $\varphi^{-1}(M)=ker p_x$. Therefore, by Lemma \ref{principal}  $\varphi^{-1}(M)$ is a principal maximal ideal. Conversely, suppose that  $\varphi$ reflects principal maximals ideals, and let $y\in Y$. Then, since $ker q_y$ is a principal maximal ideal of $B$, it follows that $\varphi^{-1}(ker q_y)=ker (q_y\circ \varphi)$ is a principal maximal ideal of $A$ and by Lemma \ref{principal}, $ker (q_y\circ \varphi)=ker p_x$ for a unique $x\in X$. Therefore, $q_y\circ \varphi=p_x$. \\
Next, we prove that $(1)\Rightarrow (4)$. Suppose that for every $y\in Y$, there exists a unique $x\in X$ such that $q_y\circ \varphi=p_x$. Then, since $q_y\circ \varphi$ is continuous for every $y\in Y$ (as it is equal to the projection $p_x$, for some $x$), then $\varphi$ is continuous.\\
As for $(4)\Rightarrow (2)$, suppose that $\varphi$ is continuous and let $M$ be a principal maximal ideal of $B$. Then, $M$ is closed and so $\varphi^{-1}(M)$ is closed. Since $A$ is compact, then $\varphi^{-1}(M)$ is compact and by Proposition \ref{principal}, $\varphi^{-1}(M)$ is  a principal maximal ideal of $A$.\\
The remaining equivalence $(3)\Leftrightarrow (4)$ follows from \cite[Corollary VII.1.7]{Jo}.
\end{proof}
\begin{rem}
Note that it follows from the conditions (2) or (3) that every algebraic isomorphism between two CHMV-algebras is automatically a topological isomorphism, i.e., a homeomorphism.
\end{rem}
Recall that an MV-algebra is called locally finite if all its finitely generated sub-MV-algebras are finite\cite{CDM} and called locally weakly finite if its finitely generated sub-MV-algebras are finite direct products of simple MV-algebras\cite{CM}. Also an MV-algebra $A$ is called hyperarchimedean if for every $a\in A$, there exists $n\geq 1$ such that $na=(n+1)a$. It is well known that the class of locally finite MV-algebras is contained in that of locally weakly finite, which in turns is contained in the class of hyperachimedean MV-algebras.\\
The next result offers a simple description of all CHMV-algebras that are locally weakly finite.
\begin{prop}\label{weak}
Given a CHMV-algebra $A:=\prod_{x\in X}\L_{n_x}$ and $\eta(A):=\{n_x:x\in X_{\text{fin}}\}$, the following assertions are equivalent. 
\begin{enumerate}
\item $A$ is locally weakly finite;
\item $A$ is hyperarchimedean;
\item $X_{\infty}$ and $\eta(A)$ are finite.
\end{enumerate}
\end{prop}
\begin{proof}
$(1) \Rightarrow (2)$: Holds for every MV-algebra as observed before the proposition.\\
$(2) \Rightarrow (3)$: Assume that $A$ is hyperarchimedean. Consider $f\in A$ defined by $f(x)=\frac{1}{n_x-1}$ for $x\in X_{\text{fin}}$ and $f(x)=0$ otherwise. Since $A$ is hyperarchimedean, there exists $n\geq 1$ such that $nf=(n+1)f$. It follows that $n\geq n_x-1$ for all $x\in X_{\text{fin}}$, that is $n_x\leq n+1$ for all $x\in X_{\text{fin}}$. Thus $\eta(A)$ is finite. In addition, if $X_{\infty}$ is infinite, then it contains a copy of $\mathbb{N}$ which we identify with $\mathbb{N}$ and write $\mathbb{N}\subseteq X_{\infty}$. Define $g\in A$ by $g(x)=\frac{1}{x}$ if $x\in \mathbb{N}$ and $g(x)=0$ otherwise. Then for every $n\geq 1$, $ng(n+1)=\frac{n}{n+1}$ and $(n+1)g(n+1)=1$. So, $ng\neq (n+1)g$ for all $n\geq 1$, which violates the fact that $A$ is hyperarchimedean. Thus, $X_{\infty}$ is finite.\\
$(3) \Rightarrow (1)$: Suppose that $X_{\infty}$ and $\eta(A)$ are finite. Then for every $f\in A$, $f(X)\subseteq (\cup_{n\in \eta{A}}\L_n)\cup f(X_{\infty})$, which is finite. Therefore, every $f\in A$ has finite range and $A$ is locally weakly finite by \cite[Theorem 3.1(iv)]{CM}.
\end{proof}
\section{CHMV-algebras versus E-multisets}
$\mathbb{CHMV}$ will denote the category of CHMV-algebras and continuous homomorphisms. We recall that a multiset is a pair $\langle X, \sigma :X\to \mathbb{N}\rangle$, where $X$ is a set and $\sigma$ is a map, and for each $x\in X$, $\sigma(x)$ is called the multiplicity of $x$. We extend this definition to extended multiset, where infinite multiplicities are allowed. More precisely, an extended multiset (e-multiset) is  a pair $\langle X, \sigma :X\to \overline{\mathbb{N}}\rangle$, where $X$ is a set and $\sigma$ is a map, and for each $x\in X$, $\sigma(x)$ is called the multiplicity of $x$. Given two G-multisets $\langle X, \sigma \rangle$ and $\langle Y, \mu \rangle$, a morphism from $\langle X, \sigma \rangle$ to $\langle Y, \mu \rangle$ is a map $\varphi: X\to Y$ such that for every $x\in X$, if $\sigma(x)$ is finite, then $\mu(\varphi(x))$ is finite and $\mu(\varphi(x))$ divides $\sigma(x)$.

The main objective of this section is to establish a categorical equivalence between the category $\mathbb{EM}$ of G-multisets and their morphisms and the category $\mathbb{CHMV}^{op}$ of CHMV-algebras and continuous homomorphisms. This equivalence will extend the equivalence between the categories of multisets and that of Stone MV-algebras obtained in \cite[Theorem 3.6]{jbn}. It should also be observed that the short proof of the equivalence in \cite[Theorem 3.6]{jbn} uses heavily the theory of Pro/Ind-completions (see \cite[Sec. VI]{Jo}) and of the connection between profinite MV-algebras and Stone MV-algebras. However, it is not clear in the present context how to deduce the equivalence using a similar technique. For this reason, we shall offer a direct and self-contained proof of the equivalence.

We start by defining two functors $\mathcal{H}: \mathbb{CHMV}^{op}\to \mathbb{EM}$ and $\mathcal{F}: \mathbb{EM}\to \mathbb{CHMV}^{op}$.
\begin{itemize}
\item[(1)] $\mathcal{H}: \mathbb{CHMV}^{op}\to \mathbb{EM}$. For any CHMV-algebra $A$, set $$\mathcal{H}_c(A):=\left\{\chi:A\to [0,1]:\chi \; \text{is a homomorphism and}\; ker\chi \; \text{is compact (maximal) ideal}\right\}$$ and $\sigma_A:\mathcal{H}_c(A)\to\overline{\mathbb{N}}$ defined by $\sigma_A(\chi)=\#\chi(A)-1$. \\ Note that $\sigma_A(\chi)=\infty$ when $\chi(A)=[0,1]$.
\begin{itemize}
\item On objects: Given a CHMV-algebra $A$, define $\mathcal{H}_c(A)=\langle \mathcal{H}_c(A), \sigma_A\rangle$.
\item On morphisms: let $\varphi$ be a homomorphism in $\mathbb{CHMV}^{op}$ from $A\to B$, that is $\varphi:B\to A$ is a continuous MV-algebras homomorphism. Define $\mathcal{H}(\varphi): \mathcal{H}_c(A) \to \mathcal{H}_c(B)$ by $\mathcal{H}(\varphi)(\chi)=\chi \circ \varphi$. Note that since $ker\chi$ is a compact maximal ideal of $A$, $ker(\chi \circ \varphi)=\varphi^{-1}(ker\chi)$, and $\varphi$ is continuous, then $ker(\chi \circ \varphi)$ is a compact maximal ideal of $B$. So, $\chi \circ \varphi\in \mathcal{H}_c(B)$ and $\mathcal{H}(\varphi)$ is well-defined. On the other hand, since isomorphic sub-MV-algebras of $[0,1]$ are equal (see for e.g., \cite[Cor. 3.5.4, Cor. 7.2.6]{C2}), it follows that for each $\chi\in \mathcal{H}_c(A)$, $\chi(A)=\L_{\#\chi(A)}$. Consequently, $\L_{\#(\chi \circ \varphi)(B)}\subseteq \L_{\#\chi(A)}$. In particular, if $\#\chi(A)-1$ is finite, so is $\#(\chi \circ \varphi)(B)-1$ and $\#(\chi \circ \varphi)(B)-1$ divides $\#\chi(A)-1$. That is, for all $\chi\in \mathcal{H}_c(A)$, $\sigma_B(\mathcal{H}(\varphi)(\chi))$ is finite whenever $\sigma_A(\chi)$ is finite and $\sigma_B(\mathcal{H}(\varphi)(\chi))$ divides $\sigma_A(\chi)$. Therefore, $\mathcal{H}(\varphi)$ is a morphism in $\mathbb{EM}$ from $\mathcal{H}_c(A) \to \mathcal{H}_c(B)$.\\
\end{itemize}
\item[(2)] $\mathcal{F}: \mathbb{EM}\to \mathbb{CHMV}^{op}$. For any G-multiset $\langle X, \sigma\rangle$, $\displaystyle\prod_{x\in X}\L_{\sigma(x)+1}$ is clearly a CHMV-algebra, that shall be denoted by $A_{X,\sigma}$.
\begin{itemize}
\item On objects: Given a multiset $\langle X, \sigma\rangle$, define $\mathcal{F}(\langle X, \sigma\rangle):=A_{X,\sigma}$.
\item On morphisms: Let $\varphi: \langle X, \sigma\rangle \to \langle Y, \mu\rangle$ be a morphism in $\mathbb{EM}$. Define $\mathcal{F}(\varphi):A_{Y,\mu}\to A_{X,\sigma}$ by $\mathcal{F}(\varphi)(f)(x)=f(\varphi(x))$ for all $f\in A_{Y,\mu}$ and all $x\in X$. To see that $\mathcal{F}(\varphi)$ is well-defined, first note that for all $f\in A_{Y,\mu}$ and all $x\in X$, $f(\varphi(x))\in \L_{\mu(\varphi(x))+1}$. On the other hand, for every $x\in X$, if $\sigma(x)$ is finite, so is $\mu(\varphi(x))$ and $\mu(\varphi(x))$ divides $\sigma(x)$. So for all $x\in X$ such that $\sigma(x)$ is finite, $ \L_{\mu(\varphi(x))+1}\subseteq \L_{\sigma(x)+1}$. The latter inclusion is obviously true when $\sigma(x)=\infty$, so $ \L_{\mu(\varphi(x))+1}\subseteq \L_{\sigma(x)+1}$ for all $x\in X$. Thus, $f(\varphi(x))\in \L_{\sigma(x)+1}$. First, it is clear that $\mathcal{H}(\varphi)$ is an MV-algebra homomorphism. In addition, let $M$ be a principal maximal ideal of $A_{X,\sigma}$, then by Lemma \ref{principal}, there exists $x_0\in X$ such that $M=M_{x_0}$. It is clear that $\mathcal{H}(\varphi)^{-1}(M_{x_0})=M_{\varphi(x_0)}$, which is a principal maximal ideal of $A_{Y,\mu}$. Therefore, $\mathcal{H}(\varphi)$ is a continuous MV-algebra homomorphism from $A_{Y,\mu} \to A_{X,\sigma}$.\\
\end{itemize}
\end{itemize}
The preceding ingredients provide the actions on objects and morphisms of two functors as formulated in the next result. 
\begin{prop}
$\mathcal{H}: \mathbb{CHMV}^{op}\to \mathbb{EM}$ and $\mathcal{F}: \mathbb{EM}\to \mathbb{CHMV}^{op}$ are functors.
\end{prop} 
\begin{proof}
This follows from the various definitions formulated above and the actual verification of the details is left to the reader.
\end{proof}
\begin{prop}\label{nattrans1}
Let $\langle X, \sigma \rangle$ be a multiset, define $\eta_X:\langle X, \sigma \rangle \to \langle \mathcal{H}_c(A_{X,\sigma}),\sigma_{A_{X,\sigma}}\rangle$ by $\eta_X(x)(f)=f(x)$, for all $x\in X$ and all $f\in A_{X,\sigma}$.\\
Then $\eta_X$ is an isomorphism in $\mathbb{EM}$.
\end{prop}
\begin{proof}
Note that for each $x\in X$, $\eta_X(x)$ is a homomorphism from $A_{X,\sigma}\to \L_{\sigma(x)+1}$, in particular $\eta_X(x)\in \mathcal{H}_c(A_{X,\sigma})$ and $\eta_X$ is well-defined. To see that $\eta_X$ is a morphism, let $x\in X$, then $\eta_X(x)(A_{X,\sigma})\subseteq \L_{\sigma(x)+1}$. Thus, $\L_{\#\eta_X(x)(A_{X,\sigma})}\subseteq \L_{\sigma(x)+1}$. Hence, for every $x\in X$, if $\sigma(x)<\infty$, then $\#\eta_X(x)(A_{X,\sigma})-1< \infty$ and $\#\eta_X(x)(A_{X,\sigma})-1$ divides $\sigma(x)$. Whence, for every $x\in X$, if $\sigma(x)<\infty$, then $\sigma_{A_{X,\sigma}}(\eta_X(x))<\infty$ and $\sigma_{A_{X,\sigma}}(\eta_X(x))$ divides $\sigma(x)$.\\
 It remains to prove that $\eta_X$ is bijective.\\
Injectivity: Let $x_1,x_2\in X$ such that $x_1\ne x_2$. Define $f\in A_{X,\sigma}$ by $f(x_1)=0$ and $f(x)=1$ for $x\ne x_1$. Then $\eta_X(x_1)(f)=0$, while $\eta_X(x_2)(f)=1$. Therefore $\eta_X(x_1)\ne \eta_X(x_2)$ and $\eta_X$ is injective.\\
Surjectivity: Let $\chi \in \mathcal{H}_c(A_{X,\sigma})$, then $ker\chi$ is a compact maximal ideal of $A_{X,\sigma}$. By Lemma \ref{principal}, there exists $x\in X$ such that $ker \chi=M_x=kerp_x$. Hence, since homomorphisms from any MV-algebra into $[0,1]$ are completely determined by their kernels (see for e.g., \cite[Sec. 4]{CDM}, the introductory paragraph), we deduce that $\chi=p_x$, and it follows that $\eta_X(x)=\chi$.\\
Thus, $\eta_X$ is an isomorphism in $\mathbb{EM}$.
\end{proof}
\begin{prop}
Let $A$ be a CHMV-algebra. Define $\varepsilon_A: A\to \displaystyle\prod_{\chi\in \mathcal{H}_c(A)}\L_{\#\chi(A)}$ by $\varepsilon_A(f)(\chi)=\chi(f)$
for all $f\in A$ and all $\chi\in  \mathcal{H}_c(A)$.\\
Then $\varepsilon_A$ is an isomorphism in $\mathbb{CHMV}^{op}$.
\end{prop}
\begin{proof}
Since $\chi(A)=\L_{\#\chi(A)}$ for all $\in \mathcal{H}_c(A)$, it follows that $\varepsilon_A$ is well-defined. In addition, let $M$ be a principal maximal ideal of $\displaystyle\prod_{\chi\in \mathcal{H}_c(A)}\L_{\#\chi(A)}$, then by Lemma \ref{principal}, there exists $\chi_0\in \mathcal{H}_c(A)$ such that $M=M_{\chi_0}$. But, it is clear that $\varepsilon_A^{-1}(M_{\chi_0})=ker \chi_0$, which is principal maximal ideal of $A$. It is straightforward to verify that $\varepsilon_A$ is a homomorphism of MV-algebras. Thus, $\varepsilon_A$ is a continuous MV-algebras homomorphism. It remains to prove that $\varepsilon_A$ is bijective.\\
Injectivity: Let $f, g\in A$ such that $\varepsilon_A(f)=\varepsilon_A(g)$, then for all $\chi\in  \mathcal{H}_c(A)$, $\chi(f)=\chi(g)$. Since $A$ is CHMV-algebra, there exists a set $X$ and $\{n_x\}_{x\in X}\subseteq \overline{\mathbb{N}}$ such that $A=\displaystyle\prod_{x\in X}\L_{n_x}$. We have $p_x(f)=p_x(g)$ for all $x\in X$, hence $f(x)=g(x)$ for all $x\in X$ and $f=g$.\\
Surjectivity: Let $g\in \displaystyle\prod_{\chi\in \mathcal{H}_c(A)}\L_{\#\chi(A)}$. Since $A$ is CHMV-algebra, there exists a set $X$ and $\{n_x\}_{x\in X}\subseteq \overline{\mathbb{N}}$ such that $A=\displaystyle\prod_{x\in X}\L_{n_x}$. Then, by Lemma \ref{principal}, $x \leftrightarrow p_x$ is a one-t-one correspondence between $X$ and $\mathcal{H}_c(A)$. Now define $f\in A$ by $f(x)=g(p_x)$. Then, it follows clearly that $\varepsilon_A(f)=g$.\\
Thus, $\varepsilon_A$ is an isomorphism in $\mathbb{CHMV}^{op}$.
\end{proof}
\begin{thm}\label{eq1}
The composite $\mathcal{H}\circ \mathcal{F}$ is naturally equivalent to the identity functor of $\mathbb{EM}$. In other words, for all G-multisets $\langle X, \sigma\rangle$, $\langle Y, \mu\rangle$ and $\varphi: \langle X, \sigma\rangle \to \langle Y, \mu\rangle$ a morphism in $\mathbb{EM}$, we have a commutative diagram
$$
\begin{CD}
\langle X, \sigma\rangle @>\varphi>> \langle Y, \mu\rangle\\
@V\eta_XVV @VV\eta_YV\\
\mathcal{H}(\mathcal{F}(\langle X, \sigma\rangle)) @>\mathcal{H}(\mathcal{F}(\varphi))>> \mathcal{H}(\mathcal{F}( \langle Y, \mu\rangle))
\end{CD}
$$
in the sense that, for each $x\in X$, $\mathcal{H}(\mathcal{F}(\varphi)(\eta_X(x))=\eta_Y(\varphi(x))$
\end{thm}
\begin{proof}
Let $x\in X$, then $\mathcal{H}(\mathcal{F}(\varphi))(\eta_X(x))=\eta_X(x)\circ \mathcal{F}(\varphi)$. For every $g\in A_{Y,\mu}$, 
$$
\begin{aligned}
(\eta_X(x)\circ \mathcal{F}(\varphi))(g)&=\eta_X(x)(\mathcal{F}(\varphi)(g))\\
&=\mathcal{F}(\varphi)(g)(x)\\
&=g(\varphi(x))\\
&=\eta_Y(\varphi(x))(g)
\end{aligned}
$$
Hence $\mathcal{H}(\mathcal{F}(\varphi)(\eta_X(x))=\eta_Y(\varphi(x))$ for all $x\in X$ as claimed.
\end{proof}
\begin{thm}\label{eq2}
The composite $\mathcal{F}\circ \mathcal{H}$ is naturally equivalent to the identity functor of $\mathbb{CHMV}^{op}$. In other words, for all all CHMV-algebras $A, B$ and $\varphi: A \to B$ a homomorphism in $\mathbb{CHMV}^{op}$, we have a commutative diagram
$$
\begin{CD}
B @>\varphi>> A\\
@V\varepsilon_BVV @VV\varepsilon_AV\\
\mathcal{F}(\mathcal{H}_c(B)) @>\mathcal{F}(\mathcal{H}(\varphi))>> \mathcal{F}(\mathcal{H}_c(A))
\end{CD}
$$
in the sense that, for each $f\in B$, $\mathcal{F}(\mathcal{H}(\varphi))(\varepsilon_B(f))=\varepsilon_A(\varphi(f))$
\end{thm}
\begin{proof}
Let $f\in B$ and $\chi \in \mathcal{H}_c(A)$, then
$$
\begin{aligned}
\mathcal{F}(\mathcal{H}(\varphi))(\varepsilon_B(f))(\chi)&=\varepsilon_B(f)(\mathcal{H}(\varphi)(\chi))\\
&=\varepsilon_B(f)(\chi\circ \varphi)\\
&=(\chi\circ \varphi)(f)\\
&=\chi(\varphi(f))\\
&=\varepsilon_A(\varphi(f))(\chi)
\end{aligned}
$$
Hence, $\mathcal{F}(\mathcal{H}(\varphi))(\varepsilon_B(f))=\varepsilon_A(\varphi(f))$ for all $f\in B$, as desired.
\end{proof}
Combining Theorem \ref{eq1} and Theorem \ref{eq2}, we obtain the anticipated duality.
\begin{cor}\label{main2}
The category $\mathbb{EM}$ of G-multisets and their morphisms is dually equivalent to the category $\mathbb{CHMV}$ of CHMV-algebras and continuous homomorphisms.
\end{cor}
\begin{rem}
The category $\mathbf{StoneMV}$ of Stone MV-algebras and continuous homomorphisms is a full subcategory of $\mathbb{CHMV}$ and the category $\mathbb{M}$ of multisets is a full subcategory of $\mathbb{EM}$. The restriction of the equivalence of Corollary \ref{main2} to $\mathbf{StoneMV}$  yields a dual equivalence between $\mathbf{StoneMV}$ and $\mathbb{M}$, which is \cite[Theorem 3.4]{jbn}.
\end{rem}
\begin{rem}
Note that for every MV-algebra $A$, one can associate the pair $\langle \text{Max}(A), \sigma_A\rangle$, where $\sigma_A:\text{Max}(A)\to \text{Sub}([0,1])$ is defined by $\sigma_A(M)=A/M$. On the category of locally weakly finite MV-algebras, this construction leads to a categorical equivalence between the category of locally weakly finite MV-algebras and the dual category of real multisets \cite{CM}. When $A$ is a CHMV-algebra, then it is known (see for e.g., \cite[Lemma 3.2]{Mund}) that for every maximal ideal $M$ of $A$, either $A/M\cong \L_n$ for some $n\geq 2$ or $A/M\cong [0,1]$. On can identify $\sigma_A$ to a map $\sigma_A:\text{Max}A\to \overline{\mathbb{N}}$. Therefore, the subcategory of CHMV-algebras described in Proposition \ref{weak}, and every principal maximal ideal $M$, the multiplicity (at $M$) of the real multiset treated in \cite{CM} coincide with the multiplicity of the extended multiset treated here.
\end{rem}
\section{Some properties of $\mathbb{CHMV}$: Urysohn-Strauss's lemma and projectivity}
Among the numerous results that are characteristics of compact Hausdorff topological spaces, the Urysohn-Strauss lemma and Gleason's Theorem are some of the most popular and well-known. In this section, we explore these results in the context of topological MV-algebras.\par
Recall that the Urysohn-Strauss's lemma for distributive compact (complete) Hausdorff topological lattices asserts that if $a\nleq b$ in such a lattice $D$, then there exists a continuous homomorphism $\varphi: D\to [0,1]$ such that $\varphi(a)=1$ and $\varphi(b)=0$\cite[Lemma VII.1.14]{Jo}. We show that the only CHMV-algebras for which the Urysohn-Strauss's lemma holds are compact Hausdorff Boolean algebras or powersets\cite{GH}.
\begin{prop}\label{us}
 The Urysohn-Strauss's lemma holds in a CHMV-algebra $A$ if and only if $A\cong \mathcal{P}(X)$ for some set $X$.
\end{prop}
\begin{proof}
Let $A=:\prod_{x\in X}\L_{n_x}$ be a CHMV-algebra and $\varphi:A\to [0,1]$ a continuous homomorphism. Then by Proposition \ref{comp-h} $\ker \varphi=\varphi^{-1}(\{0\})$ is a principal maximal ideal of $A$. It follows as in the surjectivity argument of the proof of Proposition \ref{nattrans1} that $\varphi=p_x$ for some unique $x\in X$. In other words, the only continuous homomorphisms from $A\to [0,1]$ are the natural projections. Suppose that $A$ is not a Boolean algebra, then there exists $x_0\in X$ and $0<t_0<1$ in $\L_{n_{x_0}}$. Now, if one considers $f\in A$ defined by $f(x)=0$ if $x\ne x_0$ and $f(x_0)=t_0$, then $f \nleq 0$. But by the observation above, for every continuous homomorphism $\varphi$ from $A\to [0,1]$, $\varphi(f)=0$ or $t_0$. Therefore there does not exists a continuous homomorphism from $A\to [0,1]$ sending $f$ to $1$. \par
Conversely suppose that $f\nleq g$ in $\mathbf{2}^X$. Then there exists $x_0\in X$ such that $f(x_0)=1$ and $g(x_0)=0$. Consider $p_{x_0}:\mathbf{2}^X\to 2$ the projection onto the $x_0^{th}$ factor. Then $p_{x_0}$ is a continuous homomorphism satisfying $p_{x_0}(f)=1$ and $p_{x_0}(g)=0$. 
\end{proof}
Now we turn our attention to projective objects in the category of compact Hausdorff topological spaces and and also to extremally disconnected topological MV-algebras. The consideration of these classes is motivated by the fact that for compact Hausdorff topological spaces, the two classes coincide as proved by Gleason \cite{GL}. \par In the case of topological MV-algebras, the extremally disconnected topological MV-algebras are finite MV-algebras with discrete topology as we prove next.
\begin{prop}\label{extd}
The extremally disconnected topological MV-algebras are finite MV-algebras with discrete topology.
\end{prop}
\begin{proof}
It is clear that any finite MV-algebra with discrete topology is a compact Hausdorff topological MV-algebra that is extremally disconnected.\\
Suppose that $A$ is a CHMV-algebra that is extremally disconnected. Then $A$ is a Stone MV-algebra since an extremally disconnected space is totally disconnected and by \cite[Theorem 2.3]{jbn2}, $A\cong \prod_{x\in X}\L_{n_x}$ with $2\leq n_x<\infty$ for all $x\in X$. By contradiction, assume that $X$ is infinite.  Then $X$ contains a copy of $\mathbb{N}$ which we identify with $\mathbb{N}$. Note that $A\cong \prod_{x\in \mathbb{N}}A_x\times \prod_{x\in X\setminus \mathbb{N}}A_x$, where each $A_x$ is a finite MV-chain. For each integer $n\geq 1$, define $f_n\in A$ by:
$$f_n(x)=
\left\{\begin{array}{ll}
  1 & ,\ \ \mbox{if} \ \ x\in \mathbb{N}\; \;\mbox{and} \; \; x\leq n\; \; \mbox{or}\; \; x\notin \mathbb{N}\\
 0&  ,\ \ \mbox{otherwise} 
\end{array}\right.$$
We claim that the sequence $(f_n)_n$ converges to $f\in A$, where $f(x)=1$ for all $x\in X$. Let $\mathcal{U}$ be a basis open set of $A$ with $f\in \mathcal{U}$. Write $\mathcal{U}=\prod_{x\in X}\mathcal{U}_x$ and assume that there exists a finite subset $F$ of $X$ such that $\mathcal{U}_x=A_x$ for all $x\notin F$. Since $f\in \mathcal{U}$, then $1\in \mathcal{U}_x$ for all $x\in F$. Note that if $F\subseteq X\setminus \mathbb{N}$, then $f_n\in \mathcal{U}$ for all $n\geq 1$. If $F\cap \mathbb{N} \ne \emptyset$, let $N=\text{Max}\{x\in \mathbb{N}:\mathcal{U}_x\ne A_x\}$. Then, $f_n\in \mathcal{U}$ for all $n\geq N$. Therefore, $(f_n)_n$ is a non-stationary sequence in $A$ that converges to $f$. Whence, $A$ is not extremally disconnected by \cite[Theorem. 1.3]{GL}.
\end{proof}
\begin{lem}\label{ontohomo}
Let $C:=\prod_{x\in X}\L_{n_x}$ and $B:=\prod_{y\in Y}\L_{m_y}$ and $\psi:C\to B$ be a continuous epimorphism. 
Then,\\
(1) For every $y\in Y$, there exists a unique $x_y\in X$ such that $\L_{n_{x_y}}=\L_{m_y}$;\\
(2) The map $\tau:Y\to X$ defined by $\tau(y)=x_y$ is one-to-one and ;\\
(3) For every $f\in C$ and $y\in Y$, $\psi(f)(y)=f(x_y)$.
\end{lem}
\begin{proof}For $x\in X$ and $y\in Y$, $p_x:C\to \L_{n_x}$ and $q_y:B\to \L_{m_y}$ denote the natural projections.\\
(1) Let $y\in Y$, then by Proposition \ref{comp-h},  there exists a unique $x_y\in X$ such that $q_y\circ \psi=p_{x_y}$. So, $(q_y\circ \psi)(C)=p_{x_y}(C)$ and since $\psi(C)=B$, then $\L_{n_{x_y}}=\L_{m_y}$.\\
(2) Observe from the proof of Proposition \ref{comp-h} for every $y\in Y$, $\psi^{-1}(M_y)=M_{\tau(y)}$. It follows that if $\tau(y)=\tau(y')$, then $\psi^{-1}(M_y)=\psi^{-1}(M_y')$. Hence as $\psi$ is onto, $M_y=\psi(\psi^{-1}(M_y))=\psi(\psi^{-1}(M_y'))=M_y'$. Therefore $M_y=M_y'$ and $y=y'$ as needed.\\
(3) In addition, since $q_y\circ \psi=p_{x_y}$, then $(q_y\circ \psi)(f)=p_{x_y}(f)$ for all $f\in C$. Hence,  $\psi(f)(y)=f(x_y)$  for all $f\in C$ as stated.
\end{proof}
The next result completely describe the projective compact Hausdorff topological MV-algebras.
\begin{thm}\label{proj}
Let $A$ be a compact Hausdorff MV-algebra. The following assertions are equivalent.
\begin{enumerate}
\item $A$ is projective in $\mathbb{CHMV}$;
\item $A\cong \L_2\times A'$ for some CHMV-algebra $A'$;
\item Hom$_{\mathbb{CHMV}}(A,B)\ne \emptyset$ for all $B$ in $\mathbb{CHMV}$.
\end{enumerate}
\end{thm}
\begin{proof}
$(1)\Rightarrow (2)$: Suppose that $A:=\prod_{x\in X}\L_{n_x}$ is projective in $\mathbb{CHMV}$. Consider $\psi:A\times \L_2\to A$, which is a continuous epimorhism and $1_A:A\to A$, which is the identity homomorphism. Then by the projectivity of $A$, there exists $\varphi:A\to A\times \L_2$ such that $\psi\circ \varphi=1_A$. In particular, $\pi_2\circ \varphi:A\to \L_2$ is a continuous homomorphism, which must be onto. By Proposition \ref{comp-h}, there exists $x_0\in X$ such that $\pi_2\circ \varphi=p_{x_0}$. Hence $\L_2=(\pi_2\circ \varphi)(A)=p_{x_0}(A)=\L_{n_{x_0}}$. Thus, $A\cong \L_2\times A'$ where $A'=\prod_{x\in X\setminus\{x_0\}}\L_{n_x}$.\\
$(2)\Rightarrow (1)$: We shall prove that for every CHMV-algebra $A'$, $A:=\L_2\times A'$ is projective in $\mathbb{CHMV}$. To see this, consider the diagram 
\[
\xymatrix{
  & A \ar@{.>}[dl] \ar[d]^{\varphi} \\
C \ar@{->>}[r]^\psi & B & 
}
\]
in $\mathbb{CHMV}$, where $\varphi:A\to B:=\prod_{y\in Y}\L_{m_y}$ is a continuous homomorphism and $\psi: C:=\prod_{x\in X}\L_{n_x}\to B$ a continuous epimorphism. Write $A:=\L_2\times A'=\prod_{s\in S}\L_{d_s}$, where for some fixed $s_0\in S$, $d_{s_0}=2$. \\
Define $\overline{\varphi}:A\to C$ using the notations of Lemma \ref{ontohomo}  as follows. For every $f\in A$, $\overline{\varphi}(f)(x_y)=\varphi(f)(y)$ and $\overline{\varphi}(f)(x)=f(s_0)$ if $x\neq x_y$ for all $y\in Y$. We need show that $\overline{\varphi}$ is a continuous homomoprhism and $\psi\circ \overline{\varphi}=\varphi$.\\
(i) That $\overline{\varphi}$ is a homomorphism, follows from the definition of $\overline{\varphi}$ and the fact that $\varphi$ is a homomorphism.\\
(ii) To see that $\overline{\varphi}$ is continuous, we use Proposition \ref{comp-h} and show that $\overline{\varphi}$ reflects principal maximal ideals. Let $x\in X$, we check that $\overline{\varphi}^{-1}(M_x)$ is principal. Using the various definitions, it is easy to show that if $x=x_y$ for some $y\in Y$, then $\overline{\varphi}^{-1}(M_x)=\varphi^{-1}(M_y)$, which is a principal maximal ideal of $A$ as $\varphi$ is continuous. On the other hand, if $x\ne x_y$ for all $y\in Y$, then $\overline{\varphi}^{-1}(M_x)=\{0\}\times A'$, which is a principal maximal ideal of $A$.\\
(iii) To see that $\psi\circ \overline{\varphi}=\varphi$, let $f\in A$ and $y\in Y$, then by Lemma \ref{ontohomo}, $(\psi\circ \overline{\varphi})(f)(y)=\psi(\overline{\varphi}(f))(y)=\overline{\varphi}(f)(x_y)=\varphi(f)(y)$. Hence $\psi\circ \overline{\varphi}=\varphi$ as needed.\\
$(2)\Rightarrow (3)$: Let $B:=\prod_{y\in Y}\L_{m_y}$ be in  $\mathbb{CHMV}$, and consider $\varphi: \L_2\times A'\to B$ defined by $\varphi(t,f)(y)= t$ for all $y\in Y$. Then $\varphi$ is a continuous homomorphism.\\
$(3)\Rightarrow (2)$ By (3), there exists a continuous homomorphism from $A:=\prod_{x\in X}\L_{n_x}\to \L_2$. By Lemma \ref{ontohomo}(1), there exists $x_0\in X$ such that $\L_{x_0}=\L_2$. Therefore, $A\cong \L_2\times A'$ for some CHMV-algebra $A'$.
\end{proof}
\begin{cor}
Projective Stone MV-algebras, that is projective objects in $\mathbf{StoneMV}$ are exactly the Stone MV-algebras having the 2-element Boolean algebra as quotient. 
\end{cor}
Note that when this is pushed further down, one obtains that in the category of complete atomic Boolean algebras with complete homomorphism, every algebra is projective. This is not surprising given the duality stated in the first sentence of the introduction.\\
Since the categories $\mathbb{CHMV}$ and $\mathbb{EM}$ are dually equivalent, and finite MV-algebras correspond on the other side of this duality to finite multisets, the next result follows.
\begin{cor}
Injective objects in $\mathbb{EM}$ are exactly the E-multisets with at least one element of multiplicity $1$.
\end{cor}
\section{Conclusion and Final Remarks}
Using the description of principal maximal ideals of CHMV-algebras (Proposition \ref{principals} and \ref{principal}), various characterizations continuous were obtained(Proposition \ref{comp-h}). This set up the ground to establish a duality between the category of CHMV-algebras and continuous homomophisms and that of extended multisets (Corollary \ref{main2}), which is a generalization of multisets allowing infinite multiplicities that was introduced. We also obtained that the only CHMV-algebras for which the Uryson-Strauss Lemma holds are powerset Boolean algebras (Proposition \ref{us}). Finally, we determined all the extremally disconnected CHMV-algebras, which are finite MV-algebras (Proposition \ref{extd}) and also the projective CHMV-algebras, which are those with the 2-element Boolean algebra as factor (Theorem \ref{proj}). We anticipate exploring the extension of this study to a larger subclass of semisimple MV-algebras such as strictly semisimple MV-algebras (as defined in section 2) or even complete MV-algebras.

\end{document}